\documentclass[a4paper,10pt]{article}
\usepackage[utf8]{inputenc}

\usepackage[title]{appendix}
\usepackage[english]{babel}

\usepackage{amsmath,amssymb,amsopn,amsthm}

\theoremstyle{plain}
\newtheorem{theorem}{Theorem}
\newtheorem{lemma}{Lemma}

\theoremstyle{remark}

\theoremstyle{definition}
\newtheorem{definition}{Definition}

\newcommand{\thr}[1]{\mathsf{#1}}
\newcommand{\PA}{\thr{PA}}

\renewcommand{\exp}{\mathrm{exp}}

\newbox\gnBoxA
\newdimen\gnCornerHgt
\setbox\gnBoxA=\hbox{$\ulcorner$}
\global\gnCornerHgt=\ht\gnBoxA
\newdimen\gnArgHgt
\def\gnmb #1{%
\setbox\gnBoxA=\hbox{$#1$}%
\gnArgHgt=\ht\gnBoxA%
\ifnum     \gnArgHgt<\gnCornerHgt \gnArgHgt=0pt%
\else \advance \gnArgHgt by -\gnCornerHgt%
\fi \raise\gnArgHgt\hbox{$\ulcorner$} \box\gnBoxA %
\raise\gnArgHgt\hbox{$\urcorner$}}

\begin{document}

% \frontmatter
\title{Solovay's Completeness without Fixed Points.}
\author{Fedor Pakhomov\thanks{This work is supported by the Russian Science Foundation under grant 14-50-00005.}\\
  Steklov Mathematical Institute of Russian Academy of Sciences,\\ Moscow, Russia,\\
  \texttt{pakhfn@mi.ras.ru},\\
  \texttt{http://www.mi.ras.ru/{\textasciitilde}pakhfn/}}
\date{2017}
\maketitle
\begin{abstract}In this paper we present a new proof of Solovay's theorem on arithmetical completeness of Gödel-Löb provability logic $\mathsf{GL}$. Originally, completeness of $\mathsf{GL}$ with respect to interpretation of  $\Box$ as provability in $\mathsf{PA}$ was proved by R.~Solovay in 1976. The key part of Solovay's proof was his construction of an arithmetical evaluation for a given modal formula that made the formula unprovable in $\mathsf{PA}$ if it were unprovable in $\mathsf{GL}$. The arithmetical sentences for the evaluations were constructed using certain arithmetical fixed points. The method developed by Solovay have been used for establishing similar semantics for many other logics. In our proof we develop new more explicit construction of required evaluations that doesn't use any fixed points in their definitions. To our knowledge, it is the first alternative proof of the theorem that is essentially different from Solovay's proof in this key part.  
\end{abstract}

\section{Introduction} The study of provability as a modality could be traced back to at least as early as K.~Gödel work \cite{God33}. M.H.~Löb \cite{Lob55} have proved a generalization of Gödel's Second Incompleteness Theorem that is now known as Löb's Theorem. In order to formulate his theorem Löb have stated conditions on provability predicates that are now known as Hilbert-Bernays-Löb derivability conditions. Despite Löb haven't mentioned the interpretation of a modality as a provability predicate there, his conditions essentially corresponded to the standard axiomatization of modal logic $\mathsf{K4}$. Also note that arithmetical soundness of Gödel-Löb provability logic $\mathsf{GL}$ immediately follows from Löb's Theorem. 

The axioms of modal system $\mathsf{GL}$ have first appeared in \cite{Smi63}. K.~Segerberg have shown that $\mathsf{GL}$ is Kripke-complete and moreover that it is complete with respect to the class of all finite transitive irreflexive trees \cite{Seg71}. The arithmetical completeness of the system $\mathsf{GL}$ were established by R.M.~Solovay \cite{Sol76}. Solovay have proved that a modal formula $\varphi$ is a theorem of $\mathsf{GL}$ iff for every arithmetical evaluation $f(x)$ the arithmetical sentence $f(\varphi)$ is provable in $\mathsf{PA}$.

Latter modifications of Solovay's method were used in order to prove a lot of other similar results, we will mention just few of them. G.K.~Japaridze have proved arithmetical completeness of polymodal provability logic $\mathsf{GLP}$ \cite{Jap86}. V.Yu.~Shavrukov \cite{Sha88} and A.~Berarducci \cite{Ber90} have determined the interpretability logic of $\mathsf{PA}$. 

The key part of Solovay's proof was to show that in certain sense every finite $\mathsf{GL}$-model is ``embeddable'' in arithmetic. Using the construction of ``embeddings'', it is easy to construct evaluations $f_{\varphi}(x)$ such that  $\mathsf{PA}\nvdash f_{\varphi}(\varphi)$, for all $\mathsf{GL}$-unprovable  modal formulas $\varphi$. In order to construct the ``embeddings'', Solovay have used Diagonal Lemma to define certain primitive-recursive function (Solovay function), for every finite $\mathsf{GL}$ Kripke model. Then, using the functions, Solovay have defined the sentences that constituted the ``embeddings''.

D.~de Jongh, M.~Jumelet, and F.~Montagna have shown that $\mathsf{GL}$ is complete with respect to $\Sigma_1$-provability predicates for theories $\mathsf{T}\supseteq \mathsf{I\Delta_0}+\mathsf{Exp}$ \cite{JJM91}. Their proof have avoided the use of Solovay functions, however, their construction still ``emulated'' Solovay's approach using individual sentences constructed by Diagonal Lemma.

In a discussion on FOM (Foundation of Mathematics mailing list) J.~Shipman have asked a question about important theorems that have ``essentially'' only one proof \cite{FOM09}. The example of Solovay's theorem were provided by G.~Sambin. To the author knowledge, up to the date there were no proofs of Solovay's theorem that have avoided the central idea of Solovay's proof --- the Solovay's method of constructing required sentences in terms of certain fixed points.

We note that completeness of some extensions of $\mathsf{GL}$ with respect to interpretations of $\Box$ that are similar to formalized provability were proved by the completely different methods. Solovay in his paper \cite{Sol76} have briefly mentioned a method of determining modal logics of several natural interpretations of $\Box$ in set theory, namely for the interpretations of $\Box$ as ``to be true in all transitive models'' and as ``to be true in all models $\mathbf{V}_{\kappa}$, where $\kappa$ is an inaccessible cardinal'' (there are more detailed proofs in G.~Boolos book \cite[Chapter~13]{Boo95}). A modification of the method also have been used to show completeness of wide variety of extensions of $\mathsf{GL}$ with respect to artificially defined (not $\Sigma_1$) provability-like predicates \cite{Pak16}.

In the paper we present a new approach to the proof of arithmetical completeness theorem for $\mathsf{GL}$. We introduce a different method of ``embedding'' of finite $\mathsf{GL}$ Kripke models. As the result, the completeness of $\mathsf{GL}$ is achieved with the use of evaluations given by more explicitly constructed and more ``natural'' sentences (in particular, we do not rely on Diagonal Lemma in the construction). In order to avoid potential misunderstanding, we note that despite the sentences from evaluations are given explicitly, our proof rely on Gödel's Second Incompleteness Theorem and the results by P.~Pudlák \cite{Pud86} that were proved with the use of Diagonal Lemma.

Now we will give an example of unprovable $\mathsf{GL}$-formula $\varphi$ and an evaluation $f(x)$ provided by our proof such that $\mathsf{PA}\nvdash f(\varphi)$. We consider the formula $$\varphi\mathrel{\leftrightharpoons} \Diamond v \to (\Diamond u \to \Diamond (v\land u)).$$ We use the following definitions for numerical functions in order to define the evaluation $f(x)$:
$$\exp(x)=2^x,\;\; \log(x)=\max(\{y\mid \exp(y)\le x\}\cup 0),$$
$$\exp^{\star}(x)=\underbrace{\exp(\exp(\ldots\exp(}\limits_{\mbox{\scriptsize $x$ times}}0)\ldots)),\;\;\log^{\star}(x)=\max(\{y\mid \exp^{\star}(y)\le x\}\cup 0)$$
(note that the functions $\exp^{\star}(x)$ and $\log^{\star}(x)$ are called \emph{super exponentiation} and \emph{super logarithmic} functions, respectively). The evaluation $f(x)$ is given as following:
$$f(v)\mathrel{\leftrightharpoons} \exists x(\mathsf{Prf}(x,\gnmb{0=1})\land \forall y <x (\lnot \mathsf{Prf}(y,\gnmb{0=1}))\land  \log^{\star}(x)\equiv 0\;\; (\mathrm{mod}\;\; 2)),$$
$$f(u)\mathrel{\leftrightharpoons} \exists x(\mathsf{Prf}(x,\gnmb{0=1})\land \forall y <x (\lnot \mathsf{Prf}(y,\gnmb{0=1}))\land  \log^{\star}(x)\equiv 1\;\; (\mathrm{mod}\;\; 2)).$$
%For this evaluation we show that $\mathsf{PA}\nvdash f(\varphi)$.

We note that that somewhat similar approach based on the parity of $\log^{\star}$ were used by Solovay in his letter to E.~Nelson \cite{Sol86}. Solovay proved that there are sentences  $\mathsf{F}$ and $\mathsf{G}$ such that $\mathsf{I\Delta_0+\Omega_1}+\mathsf{F}$ and $\mathsf{I\Delta_0+\Omega_1}+\mathsf{G}$  are cut-interpretable in $\mathsf{I\Delta_0+\Omega_1}$, but $\mathsf{I\Delta_0+\Omega_1}+ \mathsf{F}\land \mathsf{G}$ isn't cut-interpretable in $\mathsf{I\Delta_0+\Omega_1}$. Also, H.~Kotlarski in \cite{Kot96} have used an explicit parity-based construction of a pair of sentences in order to give an alternative proof for Rosser's Theorem.

\section{Preliminaries}
Let us first define Gödel-Löb provability logic $\mathsf{GL}$. The language of $\mathsf{GL}$ extends the language of propositional calculus with propositional constants $\top$ (truth) and $\bot$ (false) by the unary modal connective $\Box$. $\mathsf{GL}$ have the following Hilbert-style deductive system:
\begin{enumerate}
\item axiom schemes of classical propositional calculus $\mathsf{PC}$;
\item $\Box (\varphi \to \psi)\to(\Box \varphi \to \Box \psi)$;
\item $\Box (\Box \varphi \to \varphi)\to \Box \varphi$;
\item $\frac{\varphi\;\;\varphi\to\psi}{\psi}$;
\item $\frac{\varphi}{\Box\varphi}$.
\end{enumerate}
The expression $\Diamond \varphi$ is an abbreviation for $\lnot\Box\lnot \varphi$.

A set with a binary relation $(W,\prec)$ is called \emph{irreflexive transitive tree} if
\begin{enumerate}
\item $\prec$ is a transitive irreflexive relation;
\item there is an element $r\in W$ that is called the \emph{root} of $(W,\prec)$ such that the upward cone $\{a\mid r\prec a\}$ coincides with $W$;
\item for any element $w\in  W$ the restriction of $\prec$ on the downward cone $\{a\mid a\prec w\}$ is a strict well-ordering order.
\end{enumerate}
Segerberg \cite{Seg71} have shown that the logic $\mathsf{GL}$ is complete with respect to the class of all  finite irreflexive transitive trees. 

%The arithmetical sentences that we will construct in order to prove completeness of $\mathsf{GL}$ will be formulated in terms of proofs of certain sentences with the least Gödel numbers. Thus our construction depend on the details of the formalization of first-order formulas and proofs. We will use results from \cite{Pud86} and for this reason we use the approach to formalization of syntactical notions from the paper.
Our proof relies on the results by R.~Verbrugge and A.~Visser \cite{VerVis94} and indirectly on  the results by P.~Pudlák \cite{Pud86}. This results are sensitive to details of formalization of some metamathematical notions. Thus unlike some other papers, where this kind of details could be safely be left unspecified, we will need to be more careful here.

We identify syntactical expressions with binary strings. We encode binary strings by positive integers numbers. A positive integer $n$ of the form $1a_{k-1}\ldots a_0$ in binary notation encodes the binary string $a_{k-1}\ldots a_0$. We note that the binary logarithm $\log(n)$ of a number $n$ coincides with the length of the binary string that the number $n$ encodes. For a formula $\mathsf{F}$ the number $n$ that encodes $\mathsf{F}$ is known as the \emph{Gödel number} of $\mathsf{F}$.

A \emph{proof} of an arithmetical formula $\varphi$ in an arithmetical theory $\mathsf{T}$ is a list of arithmetical formulas such that it ends with $\varphi$ and every formula in the list is either an axiom of $\mathsf{T}$, or is an axiom of predicate calculus, or is obtained by inference rules from previous formulas.

We will be interested in formalization of provability in the theory $\mathsf{PA}$ and its extensions by finitely many axioms. We take the standard axiomatization of $\mathsf{PA}$ (by axioms of Robinson arithmetic $\mathsf{Q}$ and the induction schema). We consider the natural axiomatization in arithmetic of the property of a number to be the Gödel number of some axiom of $\mathsf{PA}$. For all  extensions $\mathsf{T}$ of $\mathsf{PA}$ by finitely many axioms this gives us $\Delta_0$-predicates $\mathsf{Prf}_{\mathsf{T}}(x,y)$ that are natural formalizations of ``$x$ is a proof of the formula with Gödel number $y$ in the theory $\mathsf{T}$'' that is based on the definition of the notion of proof given above. And we obtain $\Sigma_1$-provability predicates $$\mathsf{Prv}_{\mathsf{T}}(y)\leftrightharpoons \exists x\mathsf{Prf}_{\mathsf{T}}(x,y).$$

We will use effective binary numerals. The \emph{$n$-th numeral} is defined as follows:
\begin{enumerate}
\item $\underline{0}$ is the term $0$;
\item $\underline{1}$ is the term $1$;
\item $\underline{2n}$ is the term $(1+1)\mathop{\cdot}\underline{n}$;
\item $\underline{2n+1}$ is the term $(1+1)\mathop{\cdot}\underline{n}+1$.
\end{enumerate}
Clearly, the length of $\underline{n}$ is $\mathcal{O}(\log(n))$.

For an arithmetical formula $\mathsf{F}$ we denote by $\gnmb{\mathsf{F}}$ the $n$-th numeral, where $n$ is the Gödel number of the formula $\mathsf{F}$.

We denote by $\mathsf{Prv}(x)$ and $\mathsf{Prf}(x,y)$ the predicates  $\mathsf{Prv}_{\mathsf{PA}}(x)$ and $\mathsf{Prf}_{\mathsf{PA}}(x,y)$.

An \emph{arithmetical evaluation} is a function $f(x)$ from $\mathsf{GL}$ formulas to the sentences of the language of first-order arithmetic such that
\begin{enumerate}
  \item $f(\varphi\land\psi)\mathrel{\leftrightharpoons}f(\varphi)\land f(\psi)$;
  \item $f(\varphi\lor\psi)\mathrel{\leftrightharpoons}f(\varphi)\lor f(\psi)$;
  \item $f(\lnot\varphi)\mathrel{\leftrightharpoons}\lnot f(\varphi)$;
  \item $f(\varphi\to\psi)\mathrel{\leftrightharpoons} f(\varphi\to\psi)$;
  \item $f(\top)\mathrel{\leftrightharpoons} 0=0$;
  \item $f(\bot)\mathrel{\leftrightharpoons} 0=1$;
  \item $f(\Box \varphi)\mathrel{\leftrightharpoons} \mathsf{Prv}(\gnmb{f(\varphi)})$.
  \end{enumerate}
Note that an arithmetical evaluation is uniquely determined by its values on propositional variables $u,v,\ldots$.
  
We will use $\top$, $\bot$, $\Box$, and $\Diamond$ within arithmetical formulas: the expression $\top$ is an abbreviation for $0=0$, the expression $\bot$ is an abbreviation for $0=1$, the expression $\Box \mathsf{F}$ is an abbreviation for $\mathsf{Prv}(\gnmb{\mathsf{F}})$, and the expression $\Diamond\mathsf{F}$ is an abbreviation for $\lnot \mathsf{Prv}(\gnmb{\lnot\mathsf{F}})$. The expressions of the form $\Box^n \mathsf{F}$ and $\Diamond^n \mathsf{F}$ are abbreviations for $\underbrace{\Box\Box\ldots\Box}\limits_{\mbox{\scriptsize $n$ times}}\mathsf{F}$ and $\underbrace{\Diamond\Diamond\ldots\Diamond}\limits_{\mbox{\scriptsize $n$ times}}\mathsf{F}$, respectively.

\section{Proof of Solovay's Theorem}\label{proof_section}
In the section we will just give a proof of ``completeness part'' of Solovay's theorem. Soundness of the logic $\mathsf{GL}$ essentially is due to Löb \cite{Lob55} and we refer a reader to Boolos book \cite[Chapter~3]{Boo95} for a detailed proof. 
\begin{theorem} \label{Solovay_completenes} If a modal formula $\varphi$ is not provable in $\mathsf{GL}$ then there exists an arithmetical evaluation $f(x)$ such that $\PA\nvdash f(\varphi)$.
\end{theorem}

Let us fix some modal formula $\varphi$ that is not provable in $\mathsf{GL}$. By Segerberg's result \cite{Seg71}, we can find a finite transitive irreflexive tree $\mathfrak{F}=(W,\prec)$ such that $r$ is the root of $\mathfrak{F}$ and there is a model $\mathbf{M}$ on $\mathfrak{F}$ with $\mathbf{M},r\nVdash \varphi$. For all the worlds $a$ of $\mathfrak{F}$ we denote by $h(a)$ their ``height'': $$h(a)=\sup (\{0\}\cup\{h(b)+1\mid a \prec b\}).$$

Let us assign arithmetical sentences $\mathsf{C}_a$ to all the worlds $a$ of $\mathfrak{F}$. We put $\mathsf{C}_r$ to be $0=0$. We consider a non-leaf world $a$ and assign sentences $\mathsf{C}_b$ to all its immediate successors $b$. Suppose $b_0,\ldots,b_n$ are all the immediate successors of $a$. We fix some enumeration $b_0,\ldots,b_n$ such that  $h(b_n)=h(a)-1$.  For $i<n$ we put $\mathsf{C}_{b_{i}}$ to be the sentence
$$\begin{aligned}\exists x ( \mathsf{Prf}_{\mathsf{PA}+\Diamond^{h(a)-1}\top}(x,\gnmb{0=1})& \land \forall y<x (\lnot \mathsf{Prf}_{\mathsf{PA}+\Diamond^{h(a)-1}\top}(y,\gnmb{0=1}))\\ & \land \log^{\star}(x)\equiv i\;\; (\mathrm{mod}\;\; n+1)\\ & \land \exists y<\exp(\exp(x))(\mathsf{Prf}_{\mathsf{PA}+\Diamond^{h(b_i)}\top}(y,\gnmb{0=1}))).\end{aligned}$$  

The sentence $\mathsf{C}_{b_n}$ is $$\Box^{h(a)}\bot \land \bigwedge\limits_{i<n} \lnot \mathsf{C}_{b_i}.$$ Note that $\PA\vdash \lnot (\mathsf{C}_{b_i}\land \mathsf{C}_{b_j})$, for $i\ne j$ and $$\PA \vdash \Box^{h(a)}\bot \mathrel{\leftrightarrow}\bigvee\limits_{i\le n} \mathsf{C}_{b_i}.$$ We note that all $\mathsf{C}_{b_i}$ are $\mathsf{PA}$-equivalent to $\Sigma_1$-sentences: it is obvious for $i\ne n$ and $\mathsf{C}_{b_n}$ is equivalent to $\Sigma_1$-sentence since it states that there is a $\mathsf{PA}+\Diamond^{h(a)-1}\top$-proof of $0=1$ and in addition it states that the least $\mathsf{PA}+\Diamond^{h(a)-1}\top$-proof of $0=1$ satisfy certain $\Delta_0(\exp)$-property.

We assign sentences $\mathsf{F}_a$ to all the worlds $a$ of $\mathfrak{F}$. The sentence $\mathsf{F}_a$ is
$$\bigwedge\limits_{b \preceq a} \mathsf{C}_b \land \Diamond^{h(a)}\top.$$
It is easy to see that the disjunction of all $\mathsf{F}_a$'s is provable in $\mathsf{PA}$ and any conjunction $\mathsf{F}_a\land \mathsf{F}_b$, for $a\ne b$, is disprovable in $\mathsf{PA}$.

\begin{lemma}\label{diamond_lemma}For any set of worlds $A$ we have
$$\mathsf{PA}+\Box^{h(r)+1}\bot\vdash \Diamond \Big(\bigvee\limits_{a\in A}\mathsf{F}_a\Big) \mathrel{\leftrightarrow} \bigvee\limits_{b,\exists a\in A(b\prec a)}\mathsf{F}_b.$$
\end{lemma}
Let us first prove Theorem \ref{Solovay_completenes}  using Lemma \ref{diamond_lemma} and only then prove the lemma.

\begin{proof}For a variable $v$ we assign the evaluation $f(v)$:
$$\bigvee\limits_{\mathbf{M},a\Vdash v}\mathsf{F}_a.$$
By induction on the length of modal formulas $\psi$ we prove that $$\mathsf{PA}+\Box^{h(r)+1}\bot\vdash f(\psi) \mathrel{\leftrightarrow} \bigvee\limits_{\mathbf{M},a\Vdash \psi}\mathsf{F}_a.$$
The only non-trivial case for the induction step is when the topmost connective of $\psi$ is modality. Assume $\psi$ is of the form $\Box \chi$. From inductive assumption we know that $$\mathsf{PA}\vdash \Box ^{h(r)+1}\bot \to (f(\chi) \mathrel{\leftrightarrow} \bigvee\limits_{\mathbf{M},a\Vdash \chi}\mathsf{F}_a).$$ We use Lemma \ref{diamond_lemma}:
$$\begin{aligned}
\mathsf{PA}+\Box^{h(r)+1}\bot\vdash f(\Box\chi) & \mathrel{\leftrightarrow} \Box (f(\chi))\\
                                                        & \mathrel{\leftrightarrow} \Box (\Box^{h(r)+1}\bot\land f(\chi))\\
                                                        & \mathrel{\leftrightarrow} \Box (\Box^{h(r)+1}\bot \land \bigvee\limits_{\mathbf{M},a\Vdash \chi}\mathsf{F}_a)\\
                                                        & \mathrel{\leftrightarrow} \Box (\bigvee\limits_{\mathbf{M},a\Vdash \chi}\mathsf{F}_a)\\
                                                        & \mathrel{\leftrightarrow} \Box (\lnot \bigvee\limits_{\mathbf{M},a\Vdash \lnot\chi}\mathsf{F}_a)\\
                                                        & \mathrel{\leftrightarrow} \lnot \Diamond (\bigvee\limits_{\mathbf{M},a\Vdash \lnot\chi}\mathsf{F}_a)\\
                                                        & \mathrel{\leftrightarrow} \lnot \bigvee\limits_{\mathbf{M},a\Vdash \Diamond\lnot \chi}\mathsf{F}_a.\\
                                                        & \mathrel{\leftrightarrow} \bigvee\limits_{\mathbf{M},a\Vdash \Box\chi}\mathsf{F}_a.\\
\end{aligned}$$

Therefore, $$\mathsf{PA}+\Box^{h(r)+1}\bot\vdash f(\varphi)  \mathrel{\leftrightarrow} \bigvee\limits_{\mathbf{M},a\Vdash \varphi}\mathsf{F}_a.$$
Since $\mathbf{M},r\nVdash \varphi$, we have $\mathsf{PA}+\Box^{h(r)+1}\bot+\mathsf{F}_r\vdash \lnot f(\varphi)$. The sentence $\mathsf{F}_r$ is just equivalent to $\Diamond^{h(r)}\top$. Hence, by Gödel's Second Incompleteness Theorem for $\mathsf{PA}+ \Diamond^{h(r)}\top$, the theory $\mathsf{PA}+\Box^{h(r)+1}\bot+\mathsf{F}_r$ is consistent. Therefore, $\lnot f(\varphi)$ is consistent with $\mathsf{PA}$ and thus $\mathsf{PA}\nvdash f(\varphi)$.
\end{proof}

In order to prove Lemma \ref{diamond_lemma}, clearly, it will be  enough to prove the following two lemmas:
\begin{lemma}For any world $a$ from $\mathfrak{F}$, we have $$\mathsf{PA}+\Box^{h(r)+1}\bot\vdash \Diamond \mathsf{F}_a \to \bigvee \limits_{b\prec a}\mathsf{F}_b.$$
\end{lemma}
\begin{proof}
  Let us reason in $\mathsf{PA}+\Box^{h(r)+1}\bot$. Assume $\Diamond \mathsf{F}_a$. We need to prove $\bigvee \limits_{b\prec a}\mathsf{F}_b$. Let us denote by $r=c_0 \prec c_1 \prec \ldots \prec c_n=a$ the maximal chain from $r$ to $a$.  Let us find the greatest $k$ such that $\mathsf{C}_{c_k}$ holds.

  Note that for any $1\le i \le n$ the sentence $\Box^{h(c_{i-1})}\bot$ implies $\mathsf{C}_{c_i}$. Indeed, $\Box^{h(c_{i-1})}\bot$ implies that $\mathsf{C}_{c}$ for some immediate successor $c$ of $c_{i-1}$. But since $\mathsf{C}_c$ is $\Sigma_1$ and we assumed $\Diamond \mathsf{F}_a$, we would have $\Diamond (\mathsf{F}_a \land \mathsf{C}_c)$, which is possible only for $c=c_i$.

  By a simple check of cases $k=0$ and $k\ne 0$ we obtain $\Box^{h(c_k)+1}\bot$. Therefore, for all $i<k$, we have $\Box^{h(c_i)}\bot$ and hence, for all $i\le k$, the sentence $\mathsf{C}_{c_i}$ holds. From $\Box( \mathsf{F}_a\to \Diamond^{h(a)}\top)$ and $\Diamond \mathsf{F}_a$ we derive $\Diamond^{h(a)+1}\top$. Thus, $\lnot \mathsf{C}_a$ and hence $k<n$. Since $\Box^{h(c_k)}\bot$ implies $C_{c_{k+1}}$, we have $\Diamond^{h(c_k)}\top$. Therefore the sentence $\mathsf{F}_{c_k}$ holds and finally we derive $\bigvee \limits_{b\prec a}\mathsf{F}_b$.
\end{proof}

\begin{lemma} \label{constructing_models}For any worlds $a\prec b$,  we have $\mathsf{PA}+\Box^{h(r)+1}\bot\vdash \mathsf{F}_a \to \Diamond \mathsf{F}_b$.
\end{lemma}

We will use model-theoretic methods in our proof of Lemma \ref{constructing_models}. More precisely, we will need to use within $\mathsf{PA}$ some facts that we will establish using model-theoretic methods.
% There is a an approach to formalization of model theory directly in $\mathsf{PA}$ (we refer a reader to P.~Hájek and P.~Pudlák book \cite[Section~B.IV.4]{HajPud98} for more details).
There is an approach to formalization in arithmetic of results obtained by model-theoretic methods that is based on the use of the systems of the second-order arithmetic. In particular there is a well-known system $\mathsf{ACA_0}$ that is a conservative extension of $\mathsf{PA}$. We  will use the formalization of model-theoretic notions in systems of second-order arithmetic that could be found in S.~Simpson book \cite[Section~II.8, Section~IV.3]{Sim09}.

The key model-theoretic result that we use is the Injecting Inconsistencies Theorem.  We will use the version of the theorem that is a corollary of the version of the theorem that were proved by A.~Visser and R.~Verbrugge \cite[Theorem~5.1]{VerVis94}. Earlier similar results are due to P.~Hájek, R.~Solovay, J.~Krajíček, and P.~Pudlák \cite{Haj83,Sol89,KraPud89}.

\begin{definition} Suppose $\mathfrak{M}$ is a model of $\mathsf{PA}$. We denote by $\mathfrak{M}\upharpoonright a$ the structure with the domain $\{e\in \mathfrak{M}\mid \mathfrak{M}\models e\le a\}$ the constant $0$ and partial functions $S$, $+$, and $\cdot$ induced by $\mathfrak{M}$ on the domain. For two structures $\mathfrak{A}$ and $\mathfrak{B}$ with the constant $0$ and (maybe) partial functions  $S$, $+$, and $\cdot$ we write
  \begin{enumerate}
  \item $\mathfrak{A}\subseteq \mathfrak{B}$ if the domain of $\mathfrak{A}$ is a subset of the domain of $\mathfrak{B}$ and for any arithmetical term $t(x_1,\ldots,x_n)$ and elements $q_1,\ldots,q_n\in\mathfrak{A}$:
    \begin{enumerate}
    \item if $p$ is the value of $t(q_1,\ldots,q_n)$ in $\mathfrak{B}$ and $p\in\mathfrak{A}$ then the value of $t(q_1,\ldots,q_n)$ is defined in $\mathfrak{A}$ and is equal to $p$,
    \item if $p$ is the value of $t(q_1,\ldots,q_n)$ in $\mathfrak{A}$ then the value of $t(q_1,\ldots,q_n)$ is defined in $\mathfrak{B}$ and is equal to $p$;
    \end{enumerate}
      
  \item $\mathfrak{A}= \mathfrak{B}$ if  $\mathfrak{A}\subseteq \mathfrak{B}$ and $\mathfrak{B}\subseteq \mathfrak{A}$.
  \end{enumerate}
  We note that the definition actually could also be applied to models of $\mathsf{I\Delta_0}$.
\end{definition}

We will show in Appendix \ref{InjIncSec} that the following theorem is formalizable in $\mathsf{ACA}_0$:
\begin{theorem}\label{implanting_inconsistencies} Let $T$ be an extension of $\mathsf{PA}$ by finitely many axioms. Let $\mathsf{Con}_{\mathsf{T}}(x)$ denote the formula $\forall y( \log(y)\le x \to \lnot \mathsf{Prf}_{\mathsf{T}}(y,\gnmb{0=1}))$. Let $\mathfrak{M}$ be a non-standard countable model of $\mathsf{T}$. And let $q$, $p$ be nonstandard elements of $\mathfrak{M}$ such that $\mathfrak{M}\models q\le p$ and $\mathfrak{M}\models \mathsf{Con}_{\mathsf{T}}(p^k)$, for all standard $k$. Then there exists a countable model $\mathfrak{N}$ of $\mathsf{T}$ such that $p\in \mathfrak{N}$ and
  \begin{enumerate}
  \item $\mathfrak{M}\upharpoonright p =\mathfrak{N} \upharpoonright p$;
  \item $\mathfrak{M}\upharpoonright \exp(p^k) \subseteq \mathfrak{N}$, for all  standard $k$;
  \item $\mathfrak{N}\models \lnot \mathsf{Con}_{\mathsf{T}}(p^q)$;
  \item \label{preserved_consistency}  $\mathfrak{N}\models  \mathsf{Con}_{\mathsf{T}}(p^k)$, for all  standard $k$.
  \end{enumerate}

\end{theorem}

Let us now prove Lemma \ref{constructing_models} using the formalization of Theorem \ref{implanting_inconsistencies}.
\begin{proof}
  It would be enough to prove the lemma for the case when $b$ is an immediate successor of $a$. Indeed, after that we will be able to derive $\Diamond^n\mathsf{F}_b$ for any $b$, $a\prec b$, where $n$ is the length of the maximal chain from $a$ to $b$; next we could conclude that we have the required $\Diamond \mathsf{F}_b$.

  Now let us consider the case when $b$ is an immediate successor of $a$ and is $b_k$ in our fixed order $b_0,\ldots,b_n$ of the immediate successors of $a$.

  For the rest of the proof we reason in $\mathsf{ACA_0}+\mathsf{F}_a+\Box^{h(r)+1}\bot$ in order to show that we have $\Diamond \mathsf{F}_b$; since $\mathsf{ACA_0}$ is a conservative extension of $\mathsf{PA}$, this will conclude the proof.

  Since we have $\Diamond^{h(a)}\top$, we could construct a model $\mathfrak{M}$ of $\mathsf{PA}+\Diamond^{h(a)-1}\top$. Suppose $v\in \mathfrak{M}$ is the least $\mathsf{PA}+\Diamond^{h(a)-1}\top$-proof of $0=1$ in $\mathfrak{M}$, if there exists one and an arbitrary nonstandard number, otherwise. Note that since we have $\Diamond^{h(a)}\top$, the element $v$ couldn't be standard. Next we find some nonstandard $u\in \mathfrak{M}$ such that
  \begin{enumerate}
  \item $\mathfrak{M}\models \exp(\exp(u))<v$,
  \item $\mathfrak{M} \models \log^{\star} (u+1)\equiv k-1 \;\; (\mathrm{mod} \;\;n+1)$,
  \item  $\mathfrak{M} \models \log^{\star} (u)\equiv k-2 \;\; (\mathrm{mod} \;\;n+1)$.
  \end{enumerate}
  We can find $u$ with this properties since we know that the functions $\exp(x)$ and $\exp^{\star}(x)$ are total on standard natural numbers and hence we know that the functions $\log(x)$ and $\log^{\star}(x)$ map nonstandard elements to nonstandard elements in  $\mathfrak{M}$.
  
  Now we apply Theorem \ref{implanting_inconsistencies}  to the model $\mathfrak{M}$ with $p=u$ and $q=\log(u)+1$. We obtain a model $\mathfrak{M}'$ of $\mathsf{PA}+\Diamond^{h(a)-1}\top$ such that $\mathfrak{M}\upharpoonright u=\mathfrak{M}'\upharpoonright u$ and there is the least $\mathsf{PA}+\Diamond^{h(a)-1}\top$-proof $d\in \mathfrak{M}'$ of $0=1$ such that  $$\mathfrak{M}'\models u+1<u^2<\log(d)\le u^{\log(u)+1}\le \exp((\log(u)+1)^2)< \exp(u).$$ 
  Thus, $$\mathfrak{M}'\models \log^{\star} (d)\equiv k \;\; (\mathrm{mod} \;\;n+1).$$ If $h(b)=h(a)-1$, then we have constructed a model of $\mathsf{PA}+\mathsf{C}_b+\Diamond^{h(b)}\top$. 

  Assume $h(b)<h(a)-1$. Clearly, there are no $\mathsf{PA}+\Diamond^{h(b)}\top$-proofs  of $0=1$ in $\mathfrak{M}'$. We apply Theorem \ref{implanting_inconsistencies} to $\mathfrak{M}'$ with $p=d^{\log(d)+1}$ and $q=\log(d)+1$. We obtain a model $\mathfrak{M}''$ of $\mathsf{PA}+\Diamond^{h(b)}\top$ such that $$\mathfrak{M}'\upharpoonright d^{\log(d)+1}=\mathfrak{M}''\upharpoonright d^{\log(d)+1},$$ there is a $\mathsf{PA}+\Diamond^{h(b)}\top$-proof of $0=1$ in $\mathfrak{M}''$ and for the least $\mathsf{PA}+\Diamond^{h(b)}\top$-proof $e\in\mathfrak{M}''$ of $0=1$ we have
  $$\mathfrak{M}''\models \log(e)\le d^{(\log(d)+1)^2}\le \exp((\log(d)+1)^3)<\exp(d).$$
  Since $\mathfrak{M}'\upharpoonright d^{\log(d)+1}=\mathfrak{M}''\upharpoonright d^{\log(d)+1}$ and $\mathsf{Prf}(x,y)$ is a $\Delta_0$ predicate, we see that $d$ is the least $\mathsf{PA}+\Diamond^{h(a)-1}\top$-proof of $0=1$ in $\mathfrak{M}''$. Hence $\mathfrak{M}''$ is a model of $\mathsf{PA}+\mathsf{C}_b+\Diamond^{h(b)}\top$.

  Thus, under no additional assumptions, we have a model of  $\mathsf{PA}+\mathsf{C}_b+\Diamond^{h(b)}\top$. Since all $\mathsf{C}_c$, for $c\preceq a$, are $\Sigma_1$-sentences, actually we have a model of $\mathsf{PA}+\mathsf{F}_b$. Therefore, $\Diamond \mathsf{F}_b$.
\end{proof}

\section{Conclusions} In the present paper we have gave a new method of constructing arithmetical evaluations of modal formulas from a given Kripke model and proved arithmetical completeness of $\mathsf{GL}$ with respect to provability in $\mathsf{PA}$ using the method. We consider the evaluations that have been constructed in the paper to be more ``natural'' than the evaluations provided by Solovay's proof.

We proved the theorem specifically for the standard provability predicate for $\mathsf{PA}$. It is unclear to author, for what exact class of provability predicates our methods are applicable. The most essential limitation for our technique seems to be the fact that it relies on the formalized version of Theorem \ref{implanting_inconsistencies}. It seems very likely that for theories that are stronger than $\mathsf{PA}$ one could apply our method with only minor adjustments. In particular, it seems that for a general result one would need to modify $\mathsf{Prf}$-predicates while preserving $\mathsf{Prv}$-predicate (up to provable equivalence) in order to ensure that \cite[Theorem~5.1]{VerVis94} is applicable. For theories that are weaker than $\mathsf{PA}$, there are more significant problems with adopting our technique. Namely, our technique essentially relies on formalized version of the Injecting Inconsistencies Theorem. And the proofs of stronger versions of this theorem \cite{KraPud89,VerVis94} essentially rely on the Omitting Types Theorem. We have provided a  proof of the Omitting Types Theorem  in $\mathsf{ACA_0}$ in Appendix \ref{OmiTypeSec}, but it is not clear whether it could be done in weaker systems. The author is not familiar with results that calibrate reverse mathematics strength of the required version of the Omitting Types Theorem. We note that reverse mathematics analysis of other version of Omitting Types Theorem have been done by D.~Hirschfeldt, R.~Shore and T.~Slaman \cite{HSS09}, in particular from their results it follows that their version of the Omitting Types Theorem is not provable in $\mathsf{WKL_0}$ but follows from $\mathsf{RT_2^2}$ (and thus couldn't be equivalent to $\mathsf{ACA_0}$ over $\mathsf{RCA_0}$). But nevertheless, we conjecture that the same kind of evaluations as we have gave in Section~\ref{proof_section} will provide completeness of $\mathsf{GL}$ for all finitely axiomatizable extensions of $\mathsf{I\Delta_0+Exp}$.

Also, since the technique that were introduced in the paper is significantly different from Solovay's technique, it seems plausible that it may give some advantage for some open problems, for which Solovay's method have been the ``default approach'' before (see \cite{BekVis06} for open problems in provability logic).

\subsection*{Acknowledgments} I want to thank David Fernández-Duque and Albert Visser for their questions that were an important part of the reason why I have started the research on the subject. And I want to thank Paula Henk, Vladimir Yu. Shavrukov, and Albert Visser for their useful comments on an early draft of the paper. 

\bibliographystyle{alpha}
\bibliography{bibliography}

\begin{thebibliography}{dJJM91}

\bibitem[Ber90]{Ber90}
Alessandro Berarducci.
\newblock The interpretability logic of {P}eano arithmetic.
\newblock {\em The Journal of Symbolic Logic}, 55(3):1059--1089, 1990.

\bibitem[Boo95]{Boo95}
George Boolos.
\newblock {\em The Logic of Provability}.
\newblock Cambridge University Press, 1995.

\bibitem[BV06]{BekVis06}
Lev Beklemishev and Albert Visser.
\newblock {\em Problems in the Logic of Provability}, pages 77--136.
\newblock Springer New York, New York, NY, 2006.

\bibitem[CK90]{ChaKei90}
Chen~C. Chang and H.~Jerome Keisler.
\newblock {\em Model theory}, volume~73.
\newblock Elsevier, 1990.

\bibitem[dJJM91]{JJM91}
Dick de~Jongh, Marc Jumelet, and Franco Montagna.
\newblock On the proof of {S}olovay's theorem.
\newblock {\em Studia Logica: An International Journal for Symbolic Logic},
  50(1):51--69, 1991.

\bibitem[Gö33]{God33}
Kurt Gödel.
\newblock {E}in {I}nterpretation des intuitionistischen {A}ussagenkalküls.
\newblock In {\em {E}rgebnisse eines mathematischen {K}olloquiums 4}, pages
  39--40. Oxford, 1933.
\newblock Reprinted: An Interpretation of the Intuitionistic Propositional
  Calculus, Feferman, S, ed. Gödel Collected Works I publications 1929-1936.

\bibitem[Há84]{Haj83}
Petr Hájek.
\newblock On a new notion of partial conservativity.
\newblock In {\em Computation and Proof Theory}, pages 217--232. Springer,
  1984.

\bibitem[HSS09]{HSS09}
Denis Hirschfeldt, Richard Shore, and Theodore Slaman.
\newblock The atomic model theorem and type omitting.
\newblock {\em Transactions of the American Mathematical Society},
  361(11):5805--5837, 2009.

\bibitem[Jap86]{Jap86}
Giorgi~K. Japaridze.
\newblock The modal logical means of investigation of provability.
\newblock Thesisis in Philosophy, in Russian, Moscow, 1986.

\bibitem[Kot96]{Kot96}
Henryk Kotlarski.
\newblock An addition to {R}osser's theorem.
\newblock {\em The Journal of Symbolic Logic}, 61(1):285--292, 1996.

\bibitem[KP89]{KraPud89}
Jan Krajíček and Pavel Pudlák.
\newblock On the structure of initial segments of models of arithmetic.
\newblock {\em Archive for Mathematical Logic}, 28(2):91--98, 1989.

\bibitem[Lö55]{Lob55}
Martin~H. Löb.
\newblock Solution of a problem of {L}eon {H}enkin.
\newblock {\em The Journal of Symbolic Logic}, 20(02):115--118, 1955.

\bibitem[Pak16]{Pak16}
Fedor Pakhomov.
\newblock Semi-provability predicates and extensions of $\mathsf{GL}$.
\newblock In {\em 11th International Conference on Advances in Modal Logic,
  Short Presentations}, pages 110--115. 2016.

\bibitem[Pud86]{Pud86}
Pavel Pudlák.
\newblock On the length of proofs of finitistic consistency statements in first
  order theories.
\newblock In {\em Logic Colloquium 84}, pages 165--196. Amsterdam:
  North-Holland, 1986.

\bibitem[Seg71]{Seg71}
Krister Segerberg.
\newblock {\em An essay in classical modal logic}.
\newblock PhD thesis, Uppsala: Filosofiska Föreningen och Filosofiska
  Institutionen vid Uppsala Universitet, 1971.

\bibitem[Sha88]{Sha88}
Vladimir~Yu. Shavrukov.
\newblock The logic of relative interpretability over {P}eano arithmetic.
\newblock {\em Technical Report}, (5), 1988.
\newblock Moscow: Steklov Mathematical Institute (in Russian).

\bibitem[Shi09]{FOM09}
Joseph Shipman.
\newblock Only one proof, 2009.
\newblock {FOM} mailing list,\\
  http://cs.nyu.edu/pipermail/fom/2009-August/013994.html.

\bibitem[Sim09]{Sim09}
Stephen~G. Simpson.
\newblock {\em Subsystems of second order arithmetic}, volume~1.
\newblock Cambridge University Press, 2009.

\bibitem[Smi63]{Smi63}
Timothy~J. Smiley.
\newblock The logical basis of ethics.
\newblock {\em Acta Philosophica Fennica}, 16:237--246, 1963.

\bibitem[Sol76]{Sol76}
Robert~M. Solovay.
\newblock Provability interpretations of modal logic.
\newblock {\em Israel Journal of Mathematics}, 25:287--304, 1976.

\bibitem[Sol86]{Sol86}
Robert~M. Solovay, May 12, 1986.
\newblock Letter by R. Solovay to E. Nelson.

\bibitem[Sol89]{Sol89}
Robert~M. Solovay.
\newblock Injecting inconsistencies into models of {PA}.
\newblock {\em Annals of Pure and Applied Logic}, 44(1-2):101--132, 1989.

\bibitem[VV94]{VerVis94}
Rineke Verbrugge and Albert Visser.
\newblock A small reflection principle for bounded arithmetic.
\newblock {\em The Journal of Symbolic Logic}, 59(03):785--812, 1994.

\bibitem[WP89]{WilPar89}
Alex Wilkie and Jeff Paris.
\newblock On the existence of end extensions of models of bounded induction.
\newblock {\em Studies in Logic and the Foundations of Mathematics},
  126:143--161, 1989.

\end{thebibliography}

\begin{subappendices}
\renewcommand{\thesection}{\Alph{section}}%
\setcounter{section}{0} 

\section{Formalization of the Omitting Types Theorem}
\label{OmiTypeSec}

In order to formalize  Theorem \ref{implanting_inconsistencies} in $\mathsf{ACA_0}$ we will first show that the Omitting Types Theorem is formalizable in $\mathsf{ACA_0}$. We will adopt the proof from \cite{ChaKei90}. We remind a reader that we use the approach to formalization of model theory from Simpson book \cite{Sim09}.

\begin{definition}[$\mathsf{ACA}_0$] Let $\mathsf{T}$ be a first-order theory and $\Sigma=\Sigma(x_1,\ldots,x_n)$ be a set of formulas of the language of $\mathsf{T}$ that have no free variables other than $x_1,\ldots,x_n$. We say that that $\mathsf{T}$ \emph{locally omits} $\Sigma$ if for every formula $\varphi(x_1,\ldots,x_n)$ at least one of the following fails:
  \begin{enumerate}
  \item the theory $\mathsf{T}+\varphi$ is consistent;
  \item for all $\psi\in \Sigma$ we have $\mathsf{T}\vdash \forall x_1,\ldots,x_n(\varphi\to \psi)$.
  \end{enumerate}
  We say that a model $\mathfrak{M}$ of $\mathsf{T}$ \emph{omits} $\Sigma$ if for any $a_1,\ldots,a_n\in \mathfrak{M}$ there is a formula $\psi(x_1,\ldots,x_n)\in\Sigma$ such that $\mathfrak{M}\not \models \psi(a_1,\ldots,a_n)$.
\end{definition}

\begin{theorem}[$\mathsf{ACA}_0$] Suppose $\mathsf{T}$ is a consistent theory that locally omits the set of formulas $\Sigma(x_1,\ldots,x_n)$. Then there is a model $\mathfrak{M}$ of $\mathsf{T}$ that omits the set $\Sigma$.
\end{theorem}
\begin{proof}We will follow the proof of \cite[Theorem~2.2.9]{ChaKei90} but make sure that our arguments could be carried out in $\mathsf{ACA}_0$.

We will prove the theorem for $n=1$, i.e. $\Sigma=\Sigma(x)$. The case $n>1$ could be proved essentially the same way, but the notations would be more complicated.  

  We extend the language of $\mathsf{T}$ by fresh constants $c_0,c_1,\ldots$. We arrange all sentences of the extended language in a sequence $\varphi_0,\varphi_1,\ldots$ (since we work in $\mathsf{ACA}_0$ the formulas are encoded by Gödel numbers and we could arrange them by their Gödel numbers). We will construct a sequence of finite sets of sentences $$\emptyset=\mathsf{U}_0\subset \mathsf{U}_1\subset \ldots \subset \mathsf{U}_m\subset \ldots$$
  such that for every $m$ we have the following:
  \begin{enumerate}
  \item \label{U_seq_1} $\mathsf{U}_m$ is consistent with $\mathsf{T}$;
  \item \label{U_seq_2} either $\varphi_m\in \mathsf{U}_{m+1}$ or $\lnot\varphi_m\in \mathsf{U}_{m+1}$;
  \item \label{U_seq_3} if $\varphi_m$ is of the form $\exists x \psi(x)$ and $\varphi_m\in \mathsf{U}_{m+1}$ then $\psi(c_p)\in \mathsf{U}_{m+1}$, where $c_p$ is the first $c_i$ that doesn't occur in $\mathsf{U}_m$ or $\varphi_m$;
  \item \label{U_seq_4} there is a formula $\chi(x)\in \Sigma$ such that $\lnot \chi(c_m)\in\mathsf{U}_{m+1}$.   
  \end{enumerate}

  We will give the definition that will determine unique sequence  $\mathsf{U}_0,\mathsf{U}_1,\ldots$. We want to make sure that for our definition of the sequence $\mathsf{U}_0,\mathsf{U}_1,\ldots$, the property of a number $x$ to be the code of the sequence $\langle \mathsf{U}_0,\mathsf{U}_1,\ldots,\mathsf{U}_y\rangle$ is expressible by a formula without second-order quantifiers. If we will ensure this, then we will be able to construct a set that encodes the sequence $\mathsf{U}_0,\mathsf{U}_1, \ldots, \mathsf{U}_m,\ldots$ using the arithmetic comprehension.

  Let us define $\mathsf{U}_{m+1}$ in terms of $\mathsf{U}_m$. If $\varphi_m$ is consistent with $\mathsf{T}\cup \mathsf{U}_m$ then we put $\sigma_m$ to be $\varphi_m$. Otherwise we put $\sigma_m$ to be $\lnot \varphi_m$. If $\sigma_m$ is $\varphi_m$ and is of the form $\exists x \psi(x)$ then we put $\xi_m$ to be $\psi(c_p)$, where $c_p$ is the first $c_i$ that doesn't occur in $\mathsf{U}_m$ or $\varphi_m$.  Otherwise, we put $\xi_m$ to be equal to $\sigma_m$.  We choose the formula $\chi(x)$ with the smallest Gödel number such that $ \chi(x)\in\Sigma$ and $\mathsf{T}\nvdash \bigwedge \mathsf{U}_m \to \chi(c_m)$. We put $\mathsf{U}_{m+1}=\mathsf{U}_m\cup\{\xi_m,\sigma_m,\chi(c_m)\}$. 

  It is easy to see that for this definition, indeed, we could express by a formula without second-order quantifiers  the property of a number $x$ to be the code of the sequence $\langle \mathsf{U}_0,\mathsf{U}_1,\ldots,\mathsf{U}_y\rangle$.  By a trivial induction on $y$ we could prove that for every $y$ the said sequence exists and unique. Thus, we have obtained the sequence $\mathsf{U}_0, \mathsf{U}_1,\ldots,\mathsf{U}_m,\ldots$ encoded by a set.

  Now, using the definition of the sequence, we could easily prove that the sequence satisfy the conditions \ref{U_seq_1}., \ref{U_seq_2}., \ref{U_seq_3}., and \ref{U_seq_4}.

  We consider the union $\mathsf{T}\cup\bigcup\limits_{i\in \mathbb{N}}\mathsf{U}_i=\mathsf{T}'$. By condition \ref{U_seq_1}. the theory $\mathsf{T}'$ is consistent. By condition \ref{U_seq_2}. the theory $\mathsf{T}'$ is complete. By condition \ref{U_seq_3}. the theory $\mathsf{T}'$ gives the truth definition with Tarski conditions for a model with the  domain $\{c_0,c_1,\ldots\}$; this gives us a model $\mathfrak{M}$ of $\mathsf{T}'$ with the domain $\{c_0,c_1,\ldots\}$. By condition \ref{U_seq_4}. the model $\mathfrak{M}$ omits the set $\Sigma$.
\end{proof}

%Verbrugge and Visser have introduced the following principle:
\section{Formalization of the Injecting Inconsistencies Theorem}
\label{InjIncSec}

Now we are going to check that Theorem \ref{implanting_inconsistencies} is provable in $\mathsf{ACA}_0$. Below we assume that a reader is familiar with the paper \cite{VerVis94} and we will use some notions from the paper without giving the definitions here. 
%\begin{definition}
% Let $\mathsf{U}$ be an arithmetical theory. The small reflection principle for a theory $\mathsf{U}$:
%   $$\forall x (Prv_{\mathsf{U}}(\gnmb{\forall y\le \overline{x} (Prf_{\mathsf{U}}(y,\gnmb{\varphi})\to\varphi)})).$$
% \end{definition}
%We need the following formalized version of \cite[Theorem~5.1]{VerVis94} in $\mathsf{ACA}_0$ and apply it to prove Theorem \ref{implanting_inconsistencies}.
\begin{theorem}\label{VisVerII} Let $\mathsf{R}\subset \mathsf{I\Delta_0+\Omega_1}$ be a finitely axiomatizable theory. Then $\mathsf{ACA_0}$ proves the following:

  Let $\mathsf{T}\supseteq \mathsf{I\Delta_0+\Omega_1}$ be a $\Sigma_1^b$-axiomatized theory for which the small reflection principle is provable in $\mathsf{R}$. Let $\mathsf{Con}_{\mathsf{T}}(x)$ denote the formula $\forall y( \log(y)\le x \to \lnot\mathsf{Prf}_{\mathsf{T}}(y,\gnmb{0=1}))$. Let $\mathfrak{M}$ be a non-standard model of $\mathsf{T}$ and let $c$, $a$ be nonstandard elements of $\mathfrak{M}$ such that  $\mathfrak{M}\models c\le a$, $\exp(a^c)\in\mathfrak{M}$, and $\mathfrak{M}\models \mathsf{Con}_{\mathsf{T}}(a^k)$, for all standard $k$. Then there exists a model $\mathfrak{K}$ of $\mathsf{T}$ such that $a\in \mathfrak{K}$ and
  \begin{enumerate}
  \item $\mathfrak{M}\upharpoonright a =\mathfrak{K} \upharpoonright a$;
  \item $\mathfrak{M}\upharpoonright \exp(a^k) \subseteq \mathfrak{K}$, for all standard $k$;
  \item $\mathfrak{K}\models \lnot \mathsf{Con}_{\mathsf{T}}(a^c)$;
  \item for all  standard $k$ we have $\mathfrak{K}\models  \mathsf{Con}_{\mathsf{T}}(a^k)$;
  \item $\mathfrak{K}\models \exp(a^c)\downarrow$.
  \end{enumerate}
\end{theorem}
\begin{proof} Essentially, we just need to formalize the proof of \cite[Theorem~5.1]{VerVis94} in $\mathsf{ACA_0}$. The only difference between our formulation and the formulation by A.~Visser and R.~Verbrugge is that we have replaced the requirement that the small reflection principle is provable in $\mathsf{I\Delta_0+\Omega_1}$ with a stronger requirement that states that the small reflection principle is provable in $\mathsf{R}$. First, we show how to formalize the proof itself and then explain why the results used in the proof are formalizable in $\mathsf{ACA_0}$.

  The only non-trivial part of the formalization of the proof itself is the issue with the lack of truth definition for the cut $$\mathfrak{N}=\{u\in \mathfrak{M}\mid u<\exp(a^k)\mbox{, for some standard $k$}\}$$ of $\mathfrak{M}$.  However, for the purposes of the proof, it would be enough for $\mathfrak{N}$ to be a weak model (i.e. poses truth definition only for axioms, \cite[Definition~II.8.9]{Sim09}). Moreover, unlike the original proof of Visser and Verbrugge, we just need $\mathfrak{N}$ to be a weak model of $\mathsf{R}+\mathsf{B\Sigma_1}$ rather than a model of $\mathsf{B\Sigma_1+\Omega_1}$.  And since $\mathsf{R}$ is externally fixed finitely axiomatizable theory, we could create the required truth definition straightforward using arithmetical comprehension. Other parts of the proof could be formalized without any complications.

  The proof of \cite[Theorem~5.1]{VerVis94} used  A.~Wilkie and J.~Paris result \cite[Theorem~1]{WilPar89}, Pudlák results from \cite{Pud86}, and the Omitting Types Theorem. We have already formalized the Omitting Types Theorem in Appendix \ref{OmiTypeSec}. The proof of \cite[Theorem~1]{WilPar89} is trivial and could be easily formalized in $\mathsf{ACA_0}$. The technique of \cite{Pud86} is purely finitistic and thus could be easily formalized in $\mathsf{ACA_0}$.
\end{proof}
%Let us recall that the theory $\mathsf{I\Delta_0}$

Now we want to derive the formalization of Theorem \ref{implanting_inconsistencies} from Theorem \ref{VisVerII}. In order to do it, we first need to fix some finite fragment  $\mathsf{R}\subset \mathsf{I\Delta_0+\Omega_1}$. And next we need to show in $\mathsf{ACA}_0$ that all the extensions of $\mathsf{PA}$ by finitely many axioms are $\Sigma_1^b$-axiomatizable extensions of $\mathsf{I\Delta_0+\Omega_1}$ for which $\mathsf{R}$ proves the small reflection principle. Obviously,  extensions of $\mathsf{PA}$ by finitely many axioms are $\Sigma_1^b$-axiomatizable (and it could be checked in $\mathsf{ACA_0}$).

In \cite[Theorem~4.20]{VerVis94} it were established that $\mathsf{I\Delta_0+\Omega_1}$ proves small reflection principle for $\mathsf{I\Delta_0+\Omega_1}$. By inspecting the proof, it is easy to see that it is possible to use only finitely many axioms of $\mathsf{I\Delta_0+\Omega_1}$ in order to prove all the instances of the small reflection principle. Now we will indicate how to modify the proof of \cite[Theorem~4.20]{VerVis94} in order to prove in a finite fragment of $\mathsf{I\Delta_0+\Omega_1}$ all the instances of the small reflection principle for all the extensions of $\mathsf{PA}$ by finitely many axioms. Actually, the only part of the proof that should be changed is \cite[Lemma~4.16]{VerVis94} that were needed to deal with the schema of $\Delta_0$-induction schema in the case of  $\mathsf{I\Delta_0+\Omega_1}$-provability. For our adaptation we need to replace it with the analogous lemma that will deal with schema of full induction in the case of provability in $\mathsf{PA}$. This analogous lemma could be proved essentially in the same way as \cite[Lemma~4.16]{VerVis94} itself with the only difference that the last part of the proof that were reducing an instance of induction schema to an instance of $\Delta_0$-induction schema will not be needed any longer.  This concludes the proof of Theorem \ref{implanting_inconsistencies} in $\mathsf{ACA_0}$.
\end{subappendices}
\end{document}